\RequirePackage{ifluatex}

\documentclass{article}

\usepackage{amsmath}
\usepackage{amssymb}
\usepackage{amsthm}
\usepackage{mathtools}
\usepackage[latin1]{inputenc}
     
\usepackage{graphicx}

\usepackage[usenames,dvipsnames]{pstricks}
\usepackage{epsfig}
\usepackage{pst-grad} 
\usepackage{pst-plot} 

\usepackage{ifthen}
\usepackage{color}

\newcommand{\intav}[1]{\mathchoice {\mathop{\vrule width 6pt height 3 pt depth  -2.5pt
\kern -8pt \intop}\nolimits_{\kern -6pt#1}} {\mathop{\vrule width
5pt height 3  pt depth -2.6pt \kern -6pt \intop}\nolimits_{#1}}
{\mathop{\vrule width 5pt height 3 pt depth -2.6pt \kern -6pt
\intop}\nolimits_{#1}} {\mathop{\vrule width 5pt height 3 pt depth
-2.6pt \kern -6pt \intop}\nolimits_{#1}}}

\def\polhk#1{\setbox0=\hbox{#1}{\ooalign{\hidewidth\lower1.5ex\hbox{`}\hidewidth\crcr\unhbox0}}}

\def\XXint#1#2#3{{\setbox0=\hbox{$#1{#2#3}{\int}$ }
\vcenter{\hbox{$#2#3$ }}\kern-.6\wd0}}

\renewcommand{\div}{\operatorname{div}}

\newcommand{\osc}{\operatorname{osc}}

\renewcommand{\div}{\operatorname{div}}

\newcommand{\dist}{\operatorname{dist}}

\newtheorem{Theorem}{Theorem}[section]

\newtheorem{Definition}{Definition}[section]
\newtheorem{Lemma}{Lemma}[section]
\newtheorem{Corollary}{Corollary}[section]
\newtheorem{Proposition}{Proposition}[section]
\newtheorem{Remark}{Remark}[section]

\begin{document}


\title{Improved regularity for the parabolic normalized p-Laplace equation}
\author{P\^edra D. S. Andrade,\;\; Makson S. Santos}
\date{\today}
	
\maketitle
	
\begin{abstract}

\noindent We derive regularity estimates for viscosity solutions to the parabolic normalized $p$-Laplace. By using approximation methods and scaling arguments for the normalized $p$-parabolic operator, we show that the gradient of bounded viscosity solutions is locally asymptotically Lipschitz continuous when $p$ is sufficiently close to 2. In addition, we establish regularity estimates in Sobolev spaces.   
\medskip


\medskip 

\end{abstract}

\vspace{.1in}

\section{Introduction}

We study the regularity of viscosity solutions to the normalized $p$-Laplace equation
\begin{equation}\label{eq_main}
u_t(x,t) - \Delta^N_p u(x,t) = f(x,t) \;\;\;\;\;\mbox{in}\;\;\;\;\;Q_1,
\end{equation}
where $p \in (1,\infty)$, $f \in {\mathcal C}(Q_1)\cap L^\infty(Q_1)$ and $Q_1 := B_1\times(-1,0]$. We prove new and sharp regularity results for the viscosity solutions to \eqref{eq_main}, when the exponent $p$ is \emph{close} to 2. In particular, we obtain gains of regularity in Sobolev and H\"older spaces. We exploit the properties of the regularized equation
\begin{equation}\label{eq_reg}
(u_\varepsilon)_t -\Delta u_\varepsilon - (p-2)\dfrac{D^2u_\varepsilon Du_\varepsilon\cdot Du_\varepsilon}{|Du_\varepsilon|^2 + \varepsilon^2},
\end{equation}
to establish the integral estimates for the solutions of the homogeneous $p$-Laplace equation. For the regularity in H\"older spaces, we argue by approximation methods importing regularity from the {\it homogeneous heat equation} to equation \eqref{eq_main}.

The normalized $p$-Laplace operator is defined as
\[
\begin{array}{rcl}
\Delta_p^Nu & := & |Du|^{2-p}\Delta_pu =  \Delta u + (p-2)\Delta^N_\infty u \vspace{0.2cm}\\
& = & \Delta u + (p-2)\left\langle D^2u\frac{Du}{|Du|}, \frac{Du}{|Du|}\right\rangle,
\end{array}
\]
where $\Delta_pu :=|Du|^{2-p}\div(|Du|^{p-2}Du)$ is the $p$-Laplace operator. Hence, the equation \eqref{eq_main} can be seen as a uniformly parabolic model in nondivergence form, with constants $\min\{(p-1,1)\}$ and $\max\{(p-1),1\}$. We also notice that this model has a singularity on the set $\{Du = 0\}$, which implies that we can not use directly the classic $C^{1,\alpha}$ regularity theory of viscosity solutions, as in \cite{Krylov-Safonov1979, Krylov-Safonov1980, Wang1992, Imbert-Silvestre2013}. 

The normalized $p$-Laplace equation appears in many branches of mathematics, ranging from differential geometry to stochastic process. At the limit $p=1$, the equation \eqref{eq_main} is related to the level set formulation of the mean curvature flow in the work of L. C. Evans and J. Spruck \cite{Evans-Spruck}. For $1 < p < \infty$, Y. Peres and S. Sheffield in \cite{Peres-Sheffield} introduced a stochastic approach for the normalized $p$-Laplace equation using the theory of tug-of-war games with noise. For the parabolic framework we refer to the work of J. J. Manfredi, M. Parviainen and J. D. Rossi \cite{Manfredi-Parviainen-Rossi}. As $p$ approaches infinity, the equation \eqref{eq_main} finds an application in image processing; see \cite{Elmoataz-Toutain-Tenbrinck2015,Juutinen-Kawohl2006}. It is worth noticing that \cite{Kohn-Serfaty2006, Evans2007, Peres-Schramm-Sheffield-Wilson2009} address the game theoretical characterization for the borderlines cases $p=1$ and $p=\infty$. 

The normalized $p$-Laplace equation has been extensively studied by many authors. The existence and uniqueness was established in \cite{Banerjee-Garofalo2013, Does2011}. In particular, A. Banerjee and N. Garofalo in \cite{Banerjee-Garofalo2013} and K. Does in \cite{Does2011} proved Lipschitz regularity with respect to the space variable. Since the model \eqref{eq_main} is uniformly parobolic, H\"older regularity of the solutions follows from the Krylov-Safonov theory see, e.g. \cite{Krylov1987, Krylov-Safonov1980}. 

In \cite{Jin-Silvestre2017}, T. Jin and L. Silvestre showed that solutions to the homogeneous case are of class ${\mathcal C}^{1+\alpha, \frac{1+\alpha}{2}}$. In that paper, the authors follow the strategy adopted in \cite{Imbert-Silvestre2013II}, for a class of fully nonlinear degenerate equations. For the equation with a bounded right-hand side, A. Attouchi and M. Parviainen in \cite{Attouchi-Parviainen2018} also established that solutions are of class ${\mathcal C}^{1+\alpha, \frac{1+\alpha}{2}}$. More recently, the variable exponent case was studied in \cite{Fang-Zhang2021}, where Y. Fang and C. Zhang obtained H\"older regularity for the gradient of the solutions. 

Besides the regularity in H\"older spaces, many authors have studied regularity estimates in Sobolev spaces. In \cite{Hoeg-Lindqvist2020}, F. H{\o}eg and P. Lindqvist proved $W^{2,1;2}_{loc}$-regularity for the viscosity solutions, in the homogeneous setting, when $p \in \left(\frac{6}{5}, \frac{14}{5}\right)$. This range for $p$ comes from their own method, and it is not clear if it can be extended to $p \in (1, \infty)$ for $d \geq 3$. This result were improved by H. Dong, F. Peng, Y. Zhang and Y. Zhou in \cite{Dong2020} were they proved that solutions are of class $W^{2,1;q}$ for $q < 2 + \delta_{n,p}$ with $\delta_{n,p} \in (0,1)$, when $p \in (1,2)\cup\left(2, 3 + \frac{2}{d-2}\right)$.      

The purpose of this paper is to study gains of the regularity of solutions to \eqref{eq_main} when $p$ is {\it close} to 2. In this scenario, our model can be seen as a perturbation of the heat equation, as in \cite{Pimentel-Santos2020}, and then we can implement the so-called {\it approximation methods}. The approximation methods were introduced by L. Caffarelli in the seminal paper \cite{Caffarelli1989} and has been applied in a more general settings in the works of E. Teixeira, J.M. Urbano and their collaborators. We refer to the reader the papers \cite{Araujo-Teixeira-Urbano2017, DaSilva-Teixeira2017, Teixeira2014, Teixeira2014I, Teixeira-Urbano2014}, just to cite a few. See also the survey \cite{Pimentel-Santos2018}. 

We combine the approximation methods  and intrinsic scaling of the $p$-parabolic operator to obtain the desired estimate in H\"older spaces. This is the content of our first result:

\begin{Theorem}\label{theo_main}
Let $u \in {\mathcal C}(Q_1)$ be a viscosity solution to \eqref{eq_main}, with $f \in {\mathcal C}(Q_1)\cap L^\infty(Q_1)$. Given $\alpha \in (0,1)$, there exists $\varepsilon >0$, to be determined later, such that if 
\[
|p-2| \leq \varepsilon,
\] 
then $u \in {\mathcal C}^{1 + \alpha, \frac{1+\alpha}{2}}(Q_{1/2})$ and there exists a universal constant $C>0$ such that 
\[
\|u\|_{{\mathcal C}^{1 + \alpha, \frac{1+\alpha}{2}}(Q_{1/2})} \leq C\left( \|u\|_{L^{\infty}(Q_1)} + \|f\|_{L^{\infty}(Q_1)} \right).
\]
\end{Theorem}  
 
Theorem \ref{theo_main} ensures that although solutions to \eqref{eq_main} are of class $C^{{1 + \beta, \frac{1+\beta}{2}}}$ ($\beta$ can be very small), the gradient of the solutions become almost Lipschitz continuous when $p$ is sufficiently close to 2. 

\begin{Remark}
The Theorem \ref{theo_main} can be extended to a more general class of equations, namely
\begin{equation}\label{eq_gen}
u_t(x,t) - |Du|^\gamma\Delta^N_p u(x,t) = f(x,t) \;\;\;\;\;\mbox{in}\;\;\;\;\;Q_1,
\end{equation}
provided
\[
|\gamma| + |p-2| \leq \varepsilon.
\]
We refer to the reader to \cite{Imbert-Jin-Silvestre2019, Attouchi-Ruosteenoja2018, Attouchi2020, Attouchi-Ruosteenoja2020} for more details.
\end{Remark}

Our next result concerns regularity in Sobolev spaces in the homogeneous setting
\begin{equation}\label{eq_main2}
u_t(x,t) - \Delta^N_p u(x,t) = 0 \;\;\;\;\;\mbox{in}\;\;\;\;\;Q_1.
\end{equation}
Here, we make use of the regularized equation \eqref{eq_reg} and the $W^{2,1;q}$-regularity theory developed by Wang in \cite{Wang1992I} to prove our estimates. More precisely, we prove the following:

\begin{Theorem}\label{theo_main2}
Let $u \in {\mathcal C}(Q_1)$ be a viscosity solution to \eqref{eq_main2}. There exists $\varepsilon_1 >0$, to be determined later, such that if
\[
|p-2| \leq \varepsilon_1,
\] 
then $u \in W^{2,1;q}(Q_{1/2})$ for every $1 < q < \infty$. In addition, there exists a universal constant $C>0$ such that 
\[
\|u\|_{W^{2,1;q}(Q_{1/2})} \leq C\|u\|_{L^{\infty}(Q_1)}.
\]
\end{Theorem} 

The Theorem \ref{theo_main2} provides higher integrability for the viscosity solutions of \eqref{eq_main2}, with the trade off of losing the precise range on the values of $p$ for which the estimate holds true.   

\begin{Remark}
By Sobolev embeddings we observe that Theorem \ref{theo_main2} is a stronger version of Theorem \ref{theo_main} for the case $f \equiv 0$, since $\varepsilon_1$ would not depend on the H\"older exponent, see Corollary \ref{col_imp}.
\end{Remark}

The remainder of this paper is organized is follows: In Section \ref{sec_not} we fix some notations and gather a few facts used throughout the paper. In Section 3, we establish an approximation that connects our model with the heat equation. The regularity in space is subject of Section 4. We conclude the paper with the time regularity in Section 5.

\section{Notations and preliminary results}\label{sec_not}

This section puts forward elementary notations and gathers a few results used in the paper.

\subsection{Elementary notation}

Initially, we introduce some standard notations which will be used throughout the paper. In what follows $B_r \subset \mathbb{R}^d$ denotes the \emph{open ball} of radius $r$ and centered at the origin. The \emph{parabolic domain} is given by
\[
Q_r\coloneqq \{(x,t)\in \mathbb{R}^{d+1}: x \in B_r,\: t \in (-r^2,0]\} \subset\mathbb{R}^{d+1}.
\] 
We define the \emph{parabolic distance} between the points $(x_1,t_1)$ and $(x_2,t_2)$ by
	\[
	d((x_1,t_1),(x_2,t_2))\coloneqq\sqrt{|x_1-x_2|^2+|t_1-t_2|}.
	\]	
Similarly, the distance between the sets $U$ and $V$ stands for
\[
\dist(U,V)\coloneqq\inf\{d((x_1,t_1),(x_2,t_2)): (x_1, t_1) \in U, (x_2, t_2) \in V\},
\]	
where $U, V$ are subsets in $\mathbb{R}^{n+1}$.

Fix $1\leq q\leq \infty$, the \emph{parabolic Sobolev space} is defined as follows
\[
W^{2,1;q}(Q_r)\coloneqq\{u\in L^{q}(Q_r): u_t,\, Du,\, D^2u \in L^q(Q_r)\}.
\]
If $u \in W^{2,1;q}(Q_r)$, we define its norm to be
\[
\|u\|_{W^{2,1;q}(Q_r)}=\left[\|u\|^q_{L^q(Q_r)}+\|u_t\|^q_{L^q(Q_r)}+\|Du\|^q_{L^q(Q_r)}+\|D^2u\|^q_{L^q(Q_r)}\right]^{\frac{1}{q}}.
\]	
 We say that $u$ belongs to $W_{loc}^{2,1;q}(Q_r) $, if $u\in W^{2,1;q}(Q_r)$ for every $ Q'\Subset Q_r$, where $Q'\Subset Q_r$ means $\dist(Q', \partial_pQ_r)>0$, where $\partial_p$ denotes the parabolic boundary of $Q_r$.	
	
 For $0<\alpha < 1$, the \emph{parabolic H\"older space} stands for $\mathcal{C}^{\alpha,\frac{\alpha}{2}}(Q_r)$. We define its norm to be
\[
\|u\|_{\mathcal{C}^{\alpha,\frac{\alpha}{2}}(Q_r)}\coloneqq \|u\|_{L^{\infty}(Q_r)}+[u]_{\mathcal{C}^{\alpha,\frac{\alpha}{2}}(Q_r)},
\]
where $[u]_{\mathcal{C}^{\alpha,\frac{\alpha}{2}}(Q_r)}$ is the semi-norm denoted by
\[
[u]_{\mathcal{C}^{\alpha,\frac{\alpha}{2}}(Q_r)}\coloneqq \sup_{\substack{(x_1,t_1),(x_2,t_2)\in Q_r \\ (x_1,t_1)\neq(x_2,t_2)}}\frac{|u(x_1,t_1)-u(x_2,t_2)|}{ d((x_1,t_1),(x_2,t_2))^{\alpha}}.
\]
We say that $u$ is a $\alpha$-H\"older continuous with respect to the spatial variable and $\frac{\alpha}{2}$-H\"older continuous with respect to the temporal variable if its norm is finite. Similarly, we say that $u \in \mathcal{C}^{1+\alpha, \frac{1+\alpha}{2}}(Q_r)$ if there exists the spatial gradient $Du(x,t)$ for every $(x,t)$ in $Q_r$ in the classical sense and its norm
	\begin{equation*}
		\begin{aligned}
			\|u\|_{\mathcal{C}^{1+\alpha,\frac{1+\alpha}{2}}(Q_r)}&\coloneqq \|u\|_{L^{\infty}(Q_r)}+\|Du\|_{L^{\infty}(Q_r)}\\
			&+ \sup_{\substack{(x_1,t_1),(x_2,t_2)\in Q_r \\ (x_1,t_1)\neq(x_2,t_2)}}\frac{|u(x_1,t_1)-u(x_2,t_2)-Du(x_1,t_1)\cdot(x_1-x_2)|}{d((x_1,t_1),(x_2,t_2))^{1+\alpha}}.
		\end{aligned}
	\end{equation*}
is finite. It means that $Du$ is $\alpha$-H\"older continuous and $u$ is $\frac{1+\alpha}{2}$-H\"older continuous with respect to the temporal variable. For more details see \cite{CKS2000, DaSilva-Teixeira2017} and the references therein.

Finally, we define the \emph{oscillation} of the function $u$ as 
\[
\osc_{Q_1} u = \sup_{Q_1} u -\inf_{Q_1}u. 
\]


%
%

\subsection{Preliminary notions}

We start with the definition of viscosity solutions to \eqref{eq_main}.

\begin{Definition} 
Let $1<p<\infty$ and $f$ is a continuous function in $Q_1$.
We say that $u \in \mathcal{C}(Q_1)$ is a viscosity subsolution to \eqref{eq_main}. If, for every  $(x_0,t_0)\in Q_1$  and $\varphi \in {\mathcal C}^2(Q_1)$ such that $u - \varphi$ has a local maximum at $(x_0,t_0)$, we have
\[
\left\{
\begin{array}{ll}
{\varphi}_t(x_0,t_0) - \Delta_p^N\varphi(x_0,t_0)\leq f(x_0,t_0),\quad &\mbox{if}\; D\varphi(x_0,t_0)\not=0,\\
{\varphi}_t(x_0,t_0) - \Delta\varphi(x_0,t_0) - (p-2){\lambda}_{\max}(D^2 \varphi(x_0,t_0))\\
\leq f(x_0,t_0),&\mbox{if}\; D\varphi(x_0,t_0)= 0\; \mbox{and}\; p\geq 2,\\
{\varphi}_t(x_0,t_0) - \Delta\varphi(x_0,t_0) - (p-2){\lambda}_{\min}(D^2 \varphi(x_0,t_0))\\
\leq f(x_0,t_0),  &\mbox{if}\; D\varphi(x_0,t_0)= 0,\; 1<p<2.\\
\end{array}
\right.
\]

Conversely, we say that $u \in \mathcal{C}(Q_1)$ is a viscosity supersolution to  \eqref{eq_main}. If, $x_0\in Q_1$  and $\varphi \in {\mathcal C}^2(Q_1)$ such that $u - \varphi$ has a local minimum at $(x_0,t_0)$, we have
\[
\left\{
\begin{array}{ll}
{\varphi}_t(x_0,t_0) - \Delta_p^N\varphi(x_0,t_0)\geq f(x_0,t_0),\quad &\mbox{if}\; D\varphi(x_0,t_0)\not=0,\\
{\varphi}_t(x_0,t_0) - \Delta\varphi(x_0,t_0) - (p-2){\lambda}_{\min}(D^2 \varphi(x_0,t_0))\\
\geq f(x_0,t_0),&\mbox{if}\; D\varphi(x_0,t_0)= 0\; \mbox{and}\; p\geq 2,\\
{\varphi}_t(x_0,t_0) - \Delta\varphi(x_0,t_0) - (p-2){\lambda}_{\max}(D^2 \varphi(x_0,t_0))\\
\geq f(x_0,t_0),  &\mbox{if}\; D\varphi(x_0,t_0)= 0,\; 1<p<2.\\
\end{array}
\right.
\]

Whenever $u$ is a viscosity subsolution and supersolution to \eqref{eq_main} in $Q_1$, we say that $u$ is a viscosity solution to the equation.
\end{Definition}
For a general definition of this notion, we refer to the reader \cite{Attouchi-Parviainen2018}. In our arguments the scaled functions satisfy a variant of \eqref{eq_main}, namely: 

\begin{equation}\label{equationII}
u_t - \Delta u - (p - 2) \left\langle D^2 u \frac{Du + \xi}{|Du + \xi|}, \frac{Du + \xi}{|Du + \xi|} \right\rangle = f \: \: \mbox{in} \: \: Q_1, 
\end{equation}
where $\xi \in \mathbb{R}^d$ is arbitrary. On account of completeness, we proceed by stating a local compactness result used in this paper that  ensures the convergence property for the sequences.

\begin{Lemma}[Compactness]\label{compactness}
Let $u \in {\mathcal C}(Q_1)$ be a viscosity solution to \eqref{equationII}. Then for all  $r\in (0, 1)$, there exist constants $\beta \in (0, 1)$ and $C>0$ such that if $\|f\|_{L^{\infty}(Q_1)} \leq 1$  and $ \osc_{Q_1}u\leq 1$ then we have
\[
\|u \|_{\mathcal{C}^{\beta,\beta/2}(Q_r)} \leq C.
\]
\end{Lemma}
For a proof of this result, we refer the reader to \cite[Lemma 3.1]{Attouchi-Parviainen2018}. We close this section with the scaling properties of our model. Throughout the paper, we require
\begin{equation}\label{eq_scal}
\|u\|_{L^\infty(Q_1)} \leq 1 \;\;\mbox{ and } \;\; \|f\|_{L^\infty(Q_1)} \leq \varepsilon,
\end{equation} 
for some $\varepsilon$ to be determined. The conditions in \eqref{eq_scal} are not restrictive. In fact, consider the function
\[
v(x,t) = \dfrac{u(\rho x, \rho^2t)}{K},
\]
for all $0 < \rho \ll 1$ and $K > 0$. Notice that $v$ is also a viscosity solution to \eqref{eq_main} in $Q_1$, with the right-hand side 
\[
\tilde{f}(x,t) = \frac{\rho^2}{K}f(\rho x, \rho^2t).
\]  
Hence, by choosing 
\[
K = \|u\|_{L^\infty(Q_1)} + \varepsilon^{-1}\|f\|_{L^\infty(Q_1)},
\]
we can assume \eqref{eq_scal} without loss of generality.

\section{ H\"older estimates}

This section is devoted to the proof of Theorem \ref{theo_main}. As usual, constants stand for $C$ may change from line to line, and depend only on the appropriate quantities.

\subsection{Geometric Tangential Path}

First, we provide an approximation lemma relating our model with the heat equation. This lemma plays a pivotal role in the paper.

\begin{Lemma}[Approximation Lemma]\label{approximationlemma}
Let $u \in {\cal C}(Q_1)$ be a normalized viscosity solution to \eqref{equationII} with $f \in L^{\infty}(Q_1) \cap {\cal C}(Q_1)$. Given $\delta >0$ there exists $\varepsilon > 0$ such that, if 
\[
\| f \|_{L^{\infty}(Q_1)} + |p - 2 |< \varepsilon,
\] 
then we can find $h \in {\cal C}^{2, 1}(Q_{7/9})$ such that
\[
\sup_{Q_{7/9}}| u(x,t) - h(x, t)| < \delta.
\]
\end{Lemma}
\begin{proof}
We argue by contradiction. Suppose that the statement does not hold, then there are $\delta_0 >0$ and sequences $(u_j)_{j \in \mathbb{N}}$, $({\xi}_j)_{j \in \mathbb{N}}$, $(p_j)_{j \in \mathbb{N}}$ and $(f_j)_{j \in \mathbb{N}}$ satisfying
\[
\| f_j \|_{L^{\infty}(Q_1)} + |p_j - 2 |< \frac{1}{j},
\]
\begin{equation} \label{limequation1}
(u_j)_t - \Delta u_j - (p_j - 2) \left\langle D^2 u_j \frac{Du_j + {\xi}_j}{|Du_j + {\xi}_j|}, \frac{Du_j + {\xi}_j}{|Du_j + {\xi}_j|} \right\rangle = f_j \: \: \text{in} \: \: Q_1,
\end{equation}
and for every $h \in {\cal C}^{2, 1}(Q_{7/9})$ and for all $ j \in \mathbb{N}$,
\begin{equation}\label{inequality1}
\sup_{(x,t) \;\in\; Q_{7/9}}| u_j(x,t) - h(x, t)| > \delta_0.
\end{equation}

From Lemma \ref{compactness}, we have
\[
u_j \in \mathcal{C}^{\beta,\frac{\beta}{2}}(Q_{8/9}) \quad \mbox{and} \quad \|u_j \|_{\mathcal{C}^{\beta,\frac{\beta}{2}}(Q_{8/9})} \leq C,
\]
where $0<\beta<1$ and $C>0$ are constants that do not depend on $j \in \mathbb{N}$. Applying the Arzel\'a-Ascoli Theorem, there exist a subsequence $(u_j)_{j \in \mathbb{N}}$ and a continuous function $u_{\infty}$ so that $(u_j)_{j \in \mathbb{N}}$ converges uniformly to $u_{\infty}$ in $Q_{8/9}$.

We now examine two different cases. We start by considering the case in which the sequence $({\xi}_j)_{j \in\mathbb{N}}$ is bounded. Using local compactness of $\mathbb{R}^d$, up to a subsequence ${\xi}_j$ converges to ${\xi}_{\infty}$. 
Evaluating the limit in \eqref{limequation1} as $j$ approaches infinity, we obtain that $u_{\infty}$ satisfies
\begin{equation}\label{equationlim}
(u_{\infty})_t -\Delta u_{\infty} = 0\;\;\; \mbox{in}\;\;\; Q_{8/9}.
\end{equation}
Since $u_{\infty}$ is a viscosity solution of  (\ref{equationlim}), it follows that $u_{\infty}$ is of class $ {\cal C}^{2, 1}$.

On the other hand, if the sequence $({\xi}_j)_{j \in\mathbb{N}}$ is unbounded, we choose a subsequence ${\xi}_j$ such that $|{\xi}_j|$ goes to infinity for every $j$. By taking $e_j := \frac{{\xi}_j}{|{\xi}_j|}$ we have  $e_j \rightarrow e_{\infty}$, for some $e_{\infty} \in \mathbb{R}^d$. Hence, we may rewrite the equation \eqref{limequation1} as follows:
\[
(u_j)_t - \Delta u_j - (p_j - 2) \left\langle D^2 u_j \frac{Du_j |{\xi}_j|^{-1}+ e_j}{|Du_j|{\xi}_j|^{-1}+ e_j|}, \frac{Du_j |{\xi}_j|^{-1}+ e_j}{|Du_j|{\xi}_j|^{-1}+ e_j|} \right\rangle = f_j \: \: \text{in} \: \: Q_1.
\]
Applying the limit in the equation above, we get that $u_{\infty}$ solves \eqref{equationlim}. Once again, we have that $u_{\infty}$ belongs to ${\cal C}^{2, 1}(Q_{7/9})$. Finally, by taking $h = u_{\infty}$, we reach a contradiction.
\end{proof}



\subsection{Regularity in space}

We proceed with the regularity in the spatial variable.

\begin{Proposition}\label{prop_spac}
Let $u \in {\cal C}(Q_1)$ be a normalized viscosity solution of 
\[
u_t - \Delta u - (p - 2) \left\langle D^2 u \frac{Du + \xi}{|Du +  \xi|}, \frac{Du +  \xi}{|Du +  \xi|} \right\rangle = f \: \: \text{in} \: \: Q_1,
\]
for any arbitrary vector $ \xi \in \mathbb{R}^d$ and $f \in L^{\infty}(Q_1) \cap {\cal C}(Q_1)$. Then, given $\alpha \in (0, 1)$, there exists $\varepsilon> 0$ such that 
\[
 \| f\|_{L^{\infty}(Q_1)} + |p - 2|\leq \varepsilon,
\]
we can find a constant $0< \rho <1$ and a sequence of affine functions $({\ell}_n)_{n\in N}$ of the form ${\ell}_n(x, t): = a_n + b_n\cdot x$ satisfying  
\[
\sup_{ Q_{\rho^n}} |u(x, t) - {\ell}_n(x, t) | \leq {\rho}^{n (1 + \alpha)},
\] 
\[
|a_{n+1} - a_n | \leq  C {\rho}^{n(1+ \alpha)}
\]
and 
\[
|b_{n+1} - b_n |\leq C {\rho}^{n\alpha}
\]
for a constant $C>0$ and for every $n \in \mathbb{N}$.
\end{Proposition}
\begin{proof}
The result follows by induction argument. For simplicity, we split the proof into two steps.

{\it Step 1.}  Take $\delta >0$ to be determined later. Lemma \ref{approximationlemma} implies that there exists a function $h \in {\cal C}^{2, 1}{(Q_{8/9})}$ satisfying 
\[
\sup_{Q_{7/9}} |u(x, t) - h(x, t)|\leq \delta.
\]
Set 
\[
{\ell}(x, t) := h(0,0) +  Dh(0,0)\cdot x.
\]
Recall that $h(0,0)$ and  $Dh(0,0)$ are uniformly bounded by a constant. Since $h \in {\cal C}^{2, 1}{(Q_{8/9})}$, we have that there exists a constant $C>0$ such that
\[
\sup_{Q_{\rho}}|h(x, t) - \ell(x, t) | \leq C {\rho}^2,
\]
where $\rho$ is a parabolic distance between the points $(x, t)$ and $(0,0)$. As consequence from the triangular inequality, we obtain
\[
\begin{array}{ccl}
\sup_{Q_{\rho}} |u(x, t) - \ell(x, t)| & \leq & \sup_{Q_{\rho}} |u(x, t) - h(x, t)| + \sup_{Q_{\rho}} |h(x, t) - \ell(x, t)| \\
& \leq & \delta + C  {\rho}^2. \\                         
\end{array}
\]
Making universal choices, we define
\[
\rho := \left(\frac{1}{2C}\right)^{\frac{1}{1 - \alpha}} \quad \text{and} \quad \delta:= \frac{\rho^{1 + \alpha}}{2}.
\]
Notice that the universal choice of $\delta$ determines the value of $\varepsilon>0$ through the Lemma \ref{approximationlemma}. Therefore
\[
\sup_{Q_{\rho}} |u(x, t) - \ell(x, t)| \leq \rho^{1+ \alpha}.
\]
This completes the case $n=1$.

{\it Step 2.} Assume that the case $n =k$ has been verified. We shall prove the case $n = k + 1$. First, let us introduce an auxiliary function $v_k:Q_1 \rightarrow \mathbb{R}^d $ defined by
\[
v_k(x, t): = \frac{u(\rho^k x, {\rho}^{2k} t) - {\ell}_k(\rho^k x, {\rho}^{2k} t) }{{\rho}^{k (1 + \alpha)}}.
\]
Notice that $v_k$ solves the following equation:
\[
(v_k)_t - \Delta v_k - (p-2) \left\langle D^2 v_k \frac{Dv_k + {\rho}^{-k \alpha}{b}_k}{|Dv_k + {\rho}^{-k \alpha}{b}_k|}, \frac{Dv_k  + {\rho}^{-k \alpha}{b}_k}{|Dv_k +{\rho}^{-k \alpha}{b}_k|} \right \rangle = f_k \: \: \text{in} \: \: Q_1,
\]
where $f_k := \frac{1}{{\rho}^{k ( \alpha -1)}} f$. That means $ f_k \in L^{\infty}(Q_1)$, if only and if, $\alpha < 1$. By induction hypothesis, we have that $v_k$ is normalized function in $Q_{8/9}$. Hence $v_k$ satisfies the assumptions of Lemma \ref{approximationlemma}, which ensures the existence of $\tilde{h} \in \mathcal{C}^{2,1}(Q_{7/9})$, so that
\[
\sup_{Q_{7/9}}| u(x,t) - \tilde{h}(x, t)| < \delta.
\]
As a consequence of Step $1$, there exists an affine function $\tilde{\ell}$ such that
\[
\sup_{Q_{\rho}} |v_k(x, t) - \tilde{\ell}(x,t) | \leq \rho^{1+ \alpha}.
\]
Defining $\ell_{k+1}(x, t):= {\ell}_k( x, t) + \rho^{k(1+\alpha)}\tilde{\ell}(\rho^{-k} x, \rho^{-2k}t)$ yields
\[
\sup_{Q_{{\rho}^{k+1}}} |u(x, t) - \ell_{k+1}(x, t)| \leq \rho^{(k+1)(1+ \alpha)}.
\]
Also, the coefficients satisfy
\begin{equation}\label{coefficientsa}
|a_{k+1} - a_k | \leq  C {\rho}^{k(1+ \alpha)}
\end{equation}
and 
\begin{equation}\label{coefficientsb}
|b_{k+1} - b_k |\leq C {\rho}^{k\alpha}
\end{equation}
for every $k \in \mathbb{N}$, and the Proposition is concluded.
\end{proof}

\begin{proof}[Proof of Theorem \ref{theo_main}]
From \eqref{coefficientsa} and \eqref{coefficientsb}, we conclude that the sequences $(a_n)_{n \in \mathbb{N}}$ and $(b_n)_{n \in \mathbb{N}}$ are Cauchy sequences, consequently there are constants $a_{\infty}$ and $b_{\infty}$ such that
\[
\lim_{n \rightarrow \infty}a_n = a_{\infty} \quad \mbox{and} \quad \lim_{n \rightarrow \infty}b_n = b_{\infty}.
\]
Moreover, we have the estimates
\[
|a_k - a_{\infty}| \leq  C {\rho}^{k(1+ \alpha)} \quad \mbox{and} \quad |b_k  - b_{\infty}|\leq C {\rho}^{k\alpha}.
\]
Given any $0< \rho \ll 1$, let $k$ be a natural number such that $\rho^{k+1} < r < \rho^k$. Thus, we estimate from the previous computations
\[
\begin{array}{ccl}
\sup_{Q_r} |u(x, t)- {\ell}_{\infty}(x, t)| & \leq & \sup_{Q_{\rho^k}} |u(x, t) -{\ell}_k(x,t)| + \sup_{Q_{\rho^k}} | {\ell}_k(x,t) - {\ell}_{\infty}(x, t)|\vspace{0.1cm}\\
                             & \leq & {\rho}^{k (1 + \alpha)} + |a_k - a_{\infty}| + {\rho}^k|b_k - b_{\infty}|\vspace{0.1cm}\\
                              & \leq & {\rho}^{k (1 + \alpha)} + C {\rho}^{k (1 + \alpha)} +  C {\rho}^k \cdot {\rho}^{k\alpha}\vspace{0.1cm}\\
                             & \leq & C \left(\frac{1}{\rho}\right)^{(1 + \alpha)} {\rho}^{(k+1)(1 + \alpha)} \vspace{0.1cm}\\
                              &\leq & C r^{(1 + \alpha)}. \\
\end{array}
\]

To conclude the proof, we characterize the coefficients $a_{\infty}$ and $b_{\infty}$. In fact, evaluating the limit in the following inequality
\begin{equation}
\sup_{Q_{{\rho}^n}}|u(x, t) - {\ell}_n(x,t)| \leq C \rho^{n(1+ \alpha)},
\end{equation}
as $n$ approaches infinity on $(x, t)=(0,0)$, we obtain that $a_{\infty} = u(0,0)$. From \cite[Lemma A.1]{Attouchi-Parviainen2018}, we can conclude that $b_{\infty}= Du(0,0)$.
This finishes the proof.
\end{proof}

\subsection{Regularity in time}

For account of completeness we present the proof of the regularity in the time variable, see also \cite[Lemma 3.5]{Attouchi-Parviainen2018}.

\begin{Lemma}[Regularity in time]\label{time_reg}
Let $u \in {\mathcal C}(Q_1)$ be a normalized viscosity solution to \eqref{eq_main}. Assume that $\rho$ is as in Proposition \ref{prop_spac}. Given $\alpha \in (0,1)$, there exists $\varepsilon$ such that if
\[ 
\|f\|_{L^{\infty}(Q_1)} + |p-2| < \varepsilon,
\] 
then for every $t \in (-r^2,0)$
\[
|u(0,t) - u(0,0)| \leq C|t|^{\frac{1+\alpha}{2}},
\]
where $C > 0$ is a universal constant.
\end{Lemma}

\begin{proof}
For $(x,t) \in Q_r$, we set
\[
v(x,t) := u(x,t) - u(0,0) - Du(0,0)\cdot x.
\]
Proposition \ref{prop_spac} implies that
\[
|v(x_1, t) - v(x_2,t)| \leq Cr^{1+\alpha},
\]
for $x_1,x_2 \in B_r$, $t \in [-r^2,0]$, and consequently we have
\[
\mbox{osc}_{B_r}v(\cdot,t) \leq Cr^{1+\alpha} =: A. 
\]
We claim that $\mbox{osc}_{Q_r}v \leq CA + 4r^2\|f\|_{L^{\infty}(Q_1)}$. In fact, notice that $v$ solves
\[
\partial_tv - \Delta v -(p-2)\left\langle D^2 v \frac{Dv + b}{|Dv + b|}, \frac{Dv + b}{|Dv + b|} \right\rangle = f \;\;\;\mbox{ in }\;\; Q_r,
\]
where $b:=Du(0,0)$. Employing the Lemma 2.2 in \cite{Attouchi-Parviainen2018}, we obtain
\[
\mbox{osc}_{Q_r}v \leq Cr^{1+\alpha} + 4r^2\|f\|_{L^\infty(Q_1)}.
\]
Therefore, 
\[
|u(0,t) - u(0,0)| = |v(0,t)| \leq C|t|^{\frac{1+\alpha}{2}}.
\]
\end{proof}

\section{Regularity in Sobolev spaces}

In this section, we present the proof of Theorem \ref{theo_main2}. Here, we examine the homogeneous problem
\begin{equation}\label{eq_main22}
u_t(x,t) - \Delta^N_p u(x,t) = 0 \;\;\;\;\;\mbox{in}\;\;\;\;\;Q_1.
\end{equation}

Let $u \in {\mathcal C}(Q_1)$ be a viscosity solution to \eqref{eq_main22}. We consider the following regularized Dirichlet problem:
\begin{equation}\label{eq_regu}
\left\{
\begin{array}{rcl}
\displaystyle (v_\varepsilon)_t - \Delta v_\varepsilon -(p-2)\dfrac{D^2v_\varepsilon Dv_\varepsilon\cdot Dv_\varepsilon}{|Dv_\varepsilon|^2 + \varepsilon^2} & = & 0 \;\;\mbox{ in }\;\; Q_{3/4} \vspace{0.2cm}\\
v_\varepsilon & = & u \;\;\mbox{ on }\;\;\partial Q_{3/4}.
\end{array}
\right.
\end{equation}

It is well know that $v_\varepsilon$ is a classical solution (in the interior), and the gradient of $v_\varepsilon$ is uniformly bounded with respect to $\varepsilon$, see for instance \cite{Does2011}.

\begin{proof}[Proof of Theorem \ref{theo_main2}] Consider the operator
\[
F(D^2u, x,t) := -\Delta u - (p-2)\dfrac{D^2u Dv_\varepsilon\cdot Dv_\varepsilon}{|Dv_\varepsilon|^2 + \varepsilon^2}.
\]
First, notice that $F$ is a uniformly elliptic operator with constants $\lambda = \min(1, p-1)$ and $\Lambda = \max(1, p-1)$. In addition, the oscillation of $F$ given by
\[
\theta_F(x,t) := \displaystyle \sup_M\dfrac{|F(M, x, t) - F(M, 0, 0)|}{|M| + 1},
\]
satisfies
\[
\theta_F(x,t) \leq 2|p-2|.
\]
It follows from \eqref{eq_regu} that $v_\varepsilon$ solves
\[
(v_\varepsilon)_t + F(D^2v_\varepsilon, x,t) = 0.
\]
Hence, by applying \cite[Theorem 5.7]{Wang1992}, we assure that there exists a positive constant $\varepsilon_1$, such that if
\[
|p-2| \leq \varepsilon_1,
\]
we have that $v_\varepsilon \in W_{loc}^{2,1;q}(Q_1)$ for all $q \in [1, \infty)$ and there is a contant $C>0$ independent of $\varepsilon_1$ such that
\begin{equation}\label{eq_estsob}
\|v_\varepsilon\|_{W^{2,1;q}} \leq C.
\end{equation}
Since $v_\varepsilon \to u$ uniformly in compact sets, see \cite[Lemma 3.1]{Hoeg-Lindqvist2020}, we obtain that $u$ also belongs to $W_{loc}^{2,1;q}(Q_1)$ and satisfies the estimate \eqref{eq_estsob}. This finishes the proof. 
\end{proof}

The following result is a direct consequence of the Theorem \ref{theo_main2}. 
\begin{Corollary}\label{col_imp}
Let $u \in {\mathcal C}(Q_1)$ be a normalized viscosity solution of \eqref{eq_main22}. Then $ u \in {\mathcal C}_{loc}^{1+\alpha, \frac{1+\alpha}{2}}(Q_1)$, with the estimate
\[
\| u \|_{{\mathcal C}^{1+\alpha, \frac{1+\alpha}{2}}(Q_1)} \leq  C\|u\|_{L^{\infty}(Q_1)}.
\]
where $C$ is a positive constant and for every $\alpha \in (0,1)$.
\end{Corollary}

\begin{proof}
The corollary follows immediately from Theorem \ref{theo_main2} with general Sobolev inequalities in Sobolev spaces.
\end{proof}

\bigskip
{\bf Acknowledgement:} PA was supported by CAPES - Brazil. MS was partially supported by CONACyT-MEXICO Grant A1-S-48577. This study was financed in part by the Coordena\c c\~ao de Aperfei\c coamento de Pessoal de N\'ivel Superior - Brazil (CAPES) - Finance Code 001.

\bibliography{andrade_santos}

\bibliographystyle{plain}

\bigskip

\noindent\textsc{P\^edra D. S. Andrade}\\
Department of Mathematics\\
Pontifical Catholic University of Rio de Janeiro -- PUC-Rio\\
22451-900, G\'avea, Rio de Janeiro-RJ, Brazil\\
\noindent\texttt{pedra.andrade@mat.puc-rio.br}

\vspace{.15in}

\noindent\textsc{Makson S. Santos}\\
Center of Investigations in Mathematics (CIMAT)\\
36000 Guanajuato Gto - MEXICO \\
\noindent\texttt{makson.santos@cimat.mx}


\end{document}